\documentclass{amsart}
\usepackage{amsthm,amsmath,amssymb, color, graphics,enumerate}
\usepackage[dvipsnames]{xcolor}
\usepackage{tikz-cd}
\usepackage{mathrsfs}
\usepackage{adjustbox}

\newtheorem{thm}{Theorem}

\newtheorem{conj}{Conjecture}

\newtheorem{prop}{Proposition}

\newtheorem{question}{Question}
\theoremstyle{definition}

\newcommand\N{\mathbb{N}}
\newcommand\Q{\mathbb{Q}}
\newcommand\R{\mathbb{R}}
\newcommand\Z{\mathbb{Z}}

\newcommand\PLoS{{PL}_+ S^1}
\newcommand\PLoI{{PL}_+ I}

\newcommand\rot{\operatorname{rot}}

\title[PL circle homeomorphism periodic under renormalization]{A piecewise linear homeomorphism of the circle preserving
 rational points and \\
  periodic under renormalization}
\author[J. Belk]{James Belk}
\author[J. Hyde]{James Hyde}
\author[J. Tatch Moore]{Justin Tatch Moore}

\address{James Belk \\ School of Mathematics and Statistics \\ University of Glasgow \\
University Place \\
Glasgow G12 8QQ \\
United Kingdom}
\address{
James Hyde \\
Department of Mathematics and Statistics \\
Binghamton University \\
PO Box 6000 \\
Binghamton, New York 13902-6000
}

\address{Justin Tatch Moore \\ Department of Mathematics \\ Malott Hall \\ Cornell University \\ Ithaca, NY 14853-4201 USA}

\thanks{
These results were obtained in 2022 while all three authors were employed at Cornell University.
They were inspired by discussions with Michele Triestino during his
visit to Cornell in February 2022.
The authors would like to thank Isabella Liousse and Michele Triestino
for reading early drafts of this paper and offering
a number of helpful comments.
The third author was supported by NSF grant DMS-1854367.
}



\begin{document}

\begin{abstract}
We demonstrate the existence of a piecewise linear
homeomorphism $f$ of $\R/\Z$ which maps rationals to rationals,
whose slopes are powers of $\frac{2}{3}$, and whose rotation number is $\sqrt{2}-1$.
This is achieved by showing that a renormalization procedure
becomes periodic when applied to $f$.
Our construction gives a negative answer to a question of D. Calegari \cite{Calegari}.
When combined with~\cite{F-obstruction},
our result also shows that $F_{\frac{2}{3}}$ does not embed into $F$, where
$F_{\frac{2}{3}}$ is the subgroup of the Stein-Thompson group $F_{2,3}$ consisting of those elements
whose slopes are powers of $\frac{2}{3}$.
Finally, we produce some evidence
suggesting a positive answer to a variation of Calegari's question and record a number of computational
observations.
\end{abstract}

\maketitle

If $f$ is an orientation preserving homeomorphism of the circle $S^1 = \R/\Z$, Poincar\'e defined the rotation number of $f$
to be:
\[
\rot(f) :=\lim_{n \to \infty} \frac{{\widetilde{f}}^n(0)}{n}
\]
modulo 1, where $\widetilde{f}:\R \to \R$ is a lift of $f$.
He proved the following theorem.

\begin{thm}
Let $f$ be a homeomorphism of the circle.
\begin{enumerate}
    \item $\rot(f) = \frac{p}{q}$ for some relatively prime $p,q$ if and only if $f$ has a periodic point of order $q$.\smallskip

    \item if $\theta:=\rot(f)$ is irrational, then there is an order preserving surjection
    $\phi: [0,1] \to [0,1]$ 
    such that $\phi(f(t)) = \phi(t) + \theta$ modulo 1. 
    
\end{enumerate}
\end{thm}
In the setting of piecewise linear homeomorphisms of $S^1$, it is natural to wonder if the rationality of the coefficients
used in the definition of the homeomorphism would imply that the rotation number is rational.
Ghys and Sergiescu proved the following result.

\begin{thm} \cite{GhysSergiescu} \label{GhysSergiescu}
If $f \in \PLoS$ maps dyadic rationals to dyadic rationals and has slopes which are powers of $2$,
then the rotation number of $f$ is rational.
\end{thm}
\noindent
Here we recall that $\PLoS$ is the group of all piecewise linear orientation preserving homeomorphisms of the circle $S^1 = \R/\Z$,
which we will identify with $[0,1)$, equipped with a suitable topology. 

Boshernitzan \cite{dense_orbits} showed that if $0 < a,b$ and $a+b < 1$, then the homeomorphism
\[
\phi_{a,b} (t) :=
\begin{cases}
\frac{1-b}{a} t + b & \textrm{ if } 0 \leq t < a \\[3pt]
\frac{b}{1-a} (t-a) & \textrm{ if } a \leq t < 1
\end{cases}
\]
has rotation number $\frac{\log k_1}{\log k_1 - \log k_2}$ where $k_1 = \frac{1-b}{a}$ and $k_2 = \frac{b}{1-a}$.
For instance if $1 < p < q$ and $a = \frac{q-p}{p(q-1)}$ and $b = \frac{p-1}{q-1}$, then the rotation number of $\phi_{a,b}$ is $\frac{\log p}{\log q}$.
See also \cite{Liousse_Thompson-Stein}.

D. Calegari \cite{Calegari}, V. Kleptsyn\footnote{
According to email communication with Danny Calegari and Michele Triestino (in both cases on September 22, 2022),
Victor Kleptsyn obtained the result independently but did not publish the result or circulate a written proof.
}, and I. Liousse \cite{Liousse} (see also \cite{Liousse_Thompson-Stein}) each independently gave a more constructive proof of \cite{GhysSergiescu} and stated Theorem \ref{GhysSergiescu}
in the following more general form:
\begin{thm} \label{CKL}
If $f \in \PLoS$ maps rationals to rationals and has slopes which are powers of a single integer, then the
rotation number of $f$ is rational. 
\end{thm}
\noindent
Calegari asked if Theorem \ref{CKL} remained true if ``single integer'' was replaced by
``single rational'' \cite[4.6]{Calegari}.

We give a negative answer to this question.

\begin{thm} \label{main}
The rotation number of 
\[
f (t) :=
\begin{cases}
\frac{3}{2} t + \frac{3}{8} & \textrm{ if } 0 \leq t < \frac{1}{4} \\[3pt]
\frac{2}{3} t + \frac{7}{12} & \textrm{ if } \frac{1}{4} \leq t < \frac{5}{8} \\[3pt]
t - \frac{5}{8} & \textrm{ if } \frac{5}{8} \leq t < 1
\end{cases}
\]
is $\sqrt{2}-1$.
\end{thm}
\noindent
This will be achieved by showing that there is a \emph{renormalization procedure} which is periodic when applied to $f$.
To our knowledge, Theorem \ref{main} provides the first example of an
element of $\PLoS$ which maps rationals to rationals and whose rotation number is an irrational algebraic number.
The method of using periodic behavior of a renormalization operation to calculate a rotation number has its roots
in \cite{Lanford}
and the relationship between return maps for circle homeomorphisms
  and the continued fraction expansion of their rotation
numbers is at this point well known.

Computer experimentation suggests that our renomalization procedure is \emph{always} eventually periodic or terminating when applied to
a homeomorphism as in Calegari's question.
Thus we make the following conjecture.

\begin{conj} \label{quad_irrational}
If $f \in \PLoS$ maps $\Q$ to $\Q$ and has slopes which are powers of a single rational, then the rotation number
of $f$ is algebraic and has degree at most~$2$.
\end{conj}
\noindent
Toward the end of this article,
we will collect some evidence --- both theoretical and computational --- in support of this conjecture.

\subsection*{An algorithm for computing rotation numbers}

We will now describe a variation of a well-known
algorithm for computing rotation numbers of homeomorphisms of $S^1$ (see e.g. \cite{Bruin})
based on the concept of \emph{renormalization of circle homeomorphisms} which originated in \cite{FKS} \cite{ROSS}.
While our algorithm's description is somewhat more terse than the standard one, its properties do not seem to differ in any essential way.

Suppose that $f:[0,1) \to [0,1)$ is a homeomorphism of $S^1$.
If $f$ has a fixed point, define $f^*$ to be the identity function and $m_f:=\infty$. 
If $f$ does not have a fixed point, 
set $r:=f(0)$, and for each $t\in [0,r)$ let $\ell(t)>0$ be minimal such that $0\leq f^{-\ell(t)}(t) < r$. Set $m_f:= \ell(0)$, and define
\[
f^*(t) = \frac{1}{r} f^{-\ell(rt)}(rt).
\]
Note that except for using $f^{-1}$ instead of $f$ when computing the return map, this
is (an accelleration of) the Rauzy-Veech renormalization for generalized interval exchange maps; see \cite[\S3]{CunhaSmania}.
Observe that $\ell(t) = m_f$ for $0 \leq t < f^{m_f}(r)$ and $\ell(t) = m_f+1 $ for $f^{m_f}(r) \leq t < r$.    
Furthermore, if $f^*$ has no fixed points, then $f^*(t) > t$ if and only if $\ell(rt) = m_f$.

\begin{prop} \label{rot_rec}
For any homeomorphism $f$ of $S^1$ with no fixed points:
\begin{enumerate}

\item \label{f*_fp} 
if $f^*$ has a fixed point, then $\rot(f) = 1/p$ where $p$ is the period of any periodic point of $f$;

\item \label{f*_no_fp}
if $f^*$ does not have a fixed point, then
$$\rot(f) = \frac{1}{m_f + \rot(f^*)}.$$

\end{enumerate}
Thus if $\rot(f)$ has a nonterminating continued fraction expansion $[0;a_1,a_2,\ldots]$,
then $m_f=a_1$ and $\rot(f^*)$ has continued fraction expansion $[0;a_2,a_3,\ldots]$.
\end{prop}

\begin{proof}
We will only verify (\ref{f*_no_fp}) when $\rot(f)$ is rational; the other cases are
  left as an exercise.
By our assumption, $f$ has a periodic point with some period $p$.
Since all periodic points of $f$ must have period $p$ and since the solutions to $f^p(t) = t$ form a closed set
there is a minimum $t \in [0,1)$ which is a periodic point for $f$. 
Observe that any periodic orbit must intersect $[0,f(0))$, and hence $t<f(0)$.
Notice also that since $f^*$ does not have any fixed points and $t$ is minimized, $t<f^{-\ell(t)}(t)$ and therefore $\ell(t)=m_f$.  
Consequently if 
\[
p := |\{f^{-k}(t) : k \in \Z\}|
\qquad \qquad
q := |\{f^{-k}(t) : 0 \leq f^{-k}(t) < f(0) \}| 
\]
then $p = m_f q + r$ for some $0 \leq r < q$.
Moreover, 
$$r = |\{f^{-k}(t) : 0 \leq f^{-k}(t) < f^{-m_f}(t)\}|.$$
It follows that $\rot(f) = \frac{q}{p}$ and $\rot(f^*) = \frac{r}{q}$.
Dividing $p = m_f q + r$ by $p$ and substituting, we obtain
$1 = m_f \rot(f) + \rot(f) \rot(f^*)$.
Solving for $\rot(f)$, we obtain the desired equality.
\end{proof}

Notice that the previous argument shows that if $\rot(f)$ is rational and $p,q,r$ are as in the proof,
then $2r < q+r \leq p$.
Since $f^{**}$ has a periodic point of order $r$,
it follows that iteratively applying the operation $*$ to an $f$ with
rational rotation number $\frac{q}{p}$ will terminate with the identity map in at most $2 \log_2 p$ steps.
Of course we typically do not know the rotation number in advance of running the algorithm.

It is worth noting that (\ref{f*_no_fp}) of Proposition \ref{rot_rec} is not true in general if we drop
the requirement that $f^*$ has no fixed points---i.e. the $p$ in (\ref{f*_fp}) need not be $m_f$ but rather is $\ell(t)$ where $t \in [0,f(0))$
is a periodic point
of $f$ (as noted above, $m_f \leq \ell(t) \leq m_f + 1$).
  
\subsection*{A self-similar function}

In order to prove Theorem \ref{main}, it suffices to show that
$f^{**} = f$ and that $m_f = m_{f^*} = 2$, since
$\sqrt{2} - 1$ is a solution to $x= \frac{1}{2 + x}$.
Observe that $f^{-1}$ maps $[0,\frac{1}{8})$ linearly onto $[\frac{5}{8},\frac{3}{4})$ and this interval linearly onto $[\frac{1}{6},\frac{1}{4})$.
In particular, if $t \in [0,\frac{1}{8})$, $m=2$ is minimal such that $f^{-m}(t)$ is in $[0,\frac{3}{8})$.
This yields $m_f = 2$.
Moreover $f^{-2}(t) = \frac{2}{3} t + \frac{1}{6}$ on $[0,\frac{1}{8})$.
It follows that $f^*(t) = \frac{3}{2} t + \frac{3}{8}$ on $[0,\frac{1}{3})$.
Similar calculations of iterates of $f$ on the intervals $[\frac{1}{8}, \frac{5}{24})$ and $[\frac{5}{24},\frac{3}{8})$ yield a complete description of the first
return map for $f^{-1}$ on $[0,f(0)) = [0,\frac{3}{8})$.
Rescaling we obtain:
\[
f^*(t) = 
\begin{cases}
\frac{2}{3} t + \frac{4}{9} & \textrm{ if } 0 \leq t < \frac{1}{3} \\[3pt]
\frac{3}{2} t + \frac{1}{6} & \textrm{ if } \frac{1}{3} \leq t < \frac{5}{9} \\[3pt]
t - \frac{5}{9} & \textrm{ if } \frac{5}{9} \leq t < 1 
\end{cases}
\]
An analogous computation yields that $m_{f^*} = 2$ and $f^{**} = f$.

\subsection*{A simple PL homeomorphism with a complicated rational rotation number}

If $q > 0$, consider the homeomorphism $f_q$ given by exchanging the intervals $[0,\frac{1}{q+1})$ and $[\frac{1}{q+1},1)$:
$$
f_q (t) :=
\begin{cases}
qt + \frac{1}{q+1} & \textrm{ if } 0 \leq t < \frac{1}{q+1} \\[3pt]
\frac{1}{q} (t-\frac{1}{q+1}) & \textrm{ if } \frac{1}{q+1} \leq t < 1 
\end{cases} 
$$
Thus $f_q$ is the involution which maps $[0,\frac{1}{q+1})$ linearly onto its complement with slope $q$.
If $0 \leq \theta < 1$, set $f_{q,\theta} := R_\theta \circ f_q$.

The functions $f_{q,r}$ were already considered by Herman (with a different parameterization), who showed that if $\rot(f_{q,r})$ is irrational,
then there does not exist a $\sigma$-finite $F_{q,r}$-invariant measure which is absolutely continuous with respect
to Haar measure \cite[\S VI.7]{Herman}.
A routine computation shows that the function $f$ in Theorem~\ref{main} equals $f_{\frac{2}{3},\frac{1}{5}}^*$ and $m_f = 1$.
In particular $\rot(f_{\frac{2}{3},\frac{1}{5}}) = \frac{\sqrt{2}}{2}$. 
Another computation yields that the rotation number of $f_{\frac{3}{7},\frac{1}{10}}$
is the golden ratio $\frac{\sqrt{5}-1}{2}$, although the renormalization
procedure has period 6 when applied to this function.

In some cases the functions $f_{q,\theta}$ have surprisingly complex \emph{rational} rotation numbers.
For instance the rotation number of $f_{\frac{7}{8},\frac{3}{8}}$ is:
$$
\frac{668882489207594075334619723191244632191899781818066714800164040622}
{761960058189671511292372730373166431351657862332319255996727602151}
$$
Its continued fraction expansion has 147 digits after the initial 0 before terminating.
Thus while this function has a periodic point, its period exceeds $10^{65}$.
This fraction has the largest denominator of all rational values of $\rot(f_{q,r})$ when the numerators and denominators of
$q$ and $r$ are all single digits.

\subsection*{Evidence toward Conjecture \ref{quad_irrational}}

If $f \in \PLoS$, define $f^{*k}$ recursively by $f^{*0} = f$ and $f^{*(k+1)} = (f^{*k})^*$.
We conjecture that if $f \in \PLoS$ maps $\Q$ to $\Q$ and has slopes which are powers
of a single rational, then for some $k < l$, $f^{*k} = f^{*l}$.
Notice that the existence of such $k$ and $l$ is equivalent to the finiteness of 
$\{f^{*n} : n \in \N\}$.

If $f \in \PLoS$, define $B_f$ to be the set of all $t \in (0,1)$ such that
$t$ is the left endpoint of a maximal interval on which $f$ is linear.
The next proposition is a variation of the well known fact that Rauzy-Veech induction does not
increase the number of intervals of a generalized interval exchange transformation.

\begin{prop} \label{breakpoints}
For all $f \in \PLoS$, $|B_{f^*}| \leq |B_f|$.
\end{prop}

\begin{proof}
If $f$ has a fixed point, then $f^*$ is the identity function and $B_{f^*}$ is empty.
If not, it can be checked that
$\{f(0) t : t \in B_{f^*} \}$ is included in the set $X$ of all $f^k(t)$ such that
$t \in B_{f^{-1}}$ and $k \geq 0$ is minimal such that $f^k(t) \in (0,f(0))$.
In either case, the desired inequality follows.
\end{proof}

Even though Proposition \ref{breakpoints} is standard,
it is worth noting that the number of discontinuities of the derivative of $f^*$ may be greater (by one) than that for $f$---for instance
this is true when $f = f_{\frac{2}{3},\frac{1}{5}}$.

The following proposition is essentially due to Herman \cite[p.75]{Herman}
(see also \cite[3.4]{CunhaSmania}).

\begin{prop} \label{slope_prop}
For all $f \in \PLoS$, there is a $C > 1$ such that if $k \geq 0$ and $q$ is a slope of $f^{*k}$, then
$C^{-1} < q < C$.
\end{prop}

\subsection*{$F_{\frac{2}{3}}$ does not embed into Thompson's group $F$}
Let $I$ denote $[0,1]$ and $\PLoI$ denote the set of all piecewise linear orientation preserving homeomorphisms
of $I$.
If $p_1,\ldots,p_k \in \N$ are relatively prime, let $F_{p_1,\ldots,p_k}$ be the subgroup $\PLoI$ consisting of those homeomorphisms whose breakpoints lie in $\Z[\frac{1}{p_1},\ldots,\frac{1}{p_k}]$ and whose slopes are products of powers of the $p_i$'s. 
Let $F_{\frac{p}{q}}$ denote the subgroup of $F_{p,q}$ consisting of those homeomorphisms whose slopes are powers of $\frac{p}{q}$.
The groups $F_{p_1,\ldots,p_k}$
were introduced by Stein \cite{Stein} and are known as the \emph{Stein-Thompson groups}.
They generalize Richard Thompson's group $F_2$, which is often denoted $F$.
The Stein-Thompson groups, and $F$ in particular, serve as important examples in group theory.

We will conclude this paper showing that the methods of \cite{F-obstruction} can be used
to prove that $F_{\frac{2}{3}}$ does not embed into $F$.
We begin by recalling some definitions from \cite{F-obstruction}.
If $g,h\in \PLoI$ and $s \in I$ satisfy that $s < g(s) < h(s) < g(h(s))=h(g(s))$, then we define $\gamma:[s,h(s)) \to [s,h(s))$ by 
$\gamma(t) = h^{-\ell(t)}(g(t))$ where $\ell(t) \geq 0$ is such that $h^{-\ell(t)}(g(t)) \in [s,h(s))$.
If we view $[s,h(s))$ as the circle obtained from $[s,h(s)]$ by identifying $s$ and $h(s)$,
then $\gamma$ is a homeomorphism.
We say that the pair $(g,h)$ is an \emph{$F$-obstruction} if, for some choice of $s$ as above, the rotation number of $\gamma$ is irrational.

\begin{thm}\cite{F-obstruction}
If $g,h \in \PLoI$ are an $F$-obstruction, then the group generated by $g$ and $h$ does not embed into
Thompson's group $F$.
\end{thm}

Thus it suffices to show that there is a pair of elements of $\PLoI$ such that the associated $\gamma$ is topologically conjugate to
the homeomorphism $f$ in Theorem \ref{main}.
Define $g:[0,1]\to [0,1]$ by
\[
g(t):=
\begin{cases}
\frac{3}{2} t  & \textrm{ if } 0 \leq t < \frac{1}{3} \\[3pt]
\frac{2}{3} t + \frac{5}{18} & \textrm{ if } \frac{1}{3} \leq t < \frac{11}{24} \\[3pt]
t + \frac{1}{8} & \textrm{ if } \frac{11}{24} \leq t < \frac{5}{8} \\[3pt]
\frac{2}{3}t + \frac{1}{3} & \textrm{ if } \frac{5}{8} \leq t \leq 1 \\
\end{cases}
\]
Define 
\[
h(t) :=
\begin{cases}
\frac{27}{8} t & \textrm{ if } 0 \leq t < \frac{1}{54} \\[3pt]
\frac{9}{4} t + \frac{1}{48} & \textrm{ if } \frac{1}{54} \leq t < \frac{1}{4} \\[3pt]
t + \frac{1}{3} & \textrm{ if } \frac{1}{4} \leq t < \frac{7}{12} \\[3pt]
\frac{16}{81} t + \frac{779}{972} & \textrm{ if } \frac{7}{12} \leq t < \frac{95}{96} \\[3pt]
\frac{8}{27} t + \frac{19}{27} & \textrm{ if } \frac{95}{96} \leq t \leq 1 \\
\end{cases}
\]
Note that 
$$\frac{1}{4} < g\Bigl(\frac{1}{4}\Bigr) = \frac{3}{8} < h\Bigl(\frac{1}{4}\Bigr) = \frac{7}{12} < g\Bigl(h\Bigl(\frac{1}{4}\Bigr)\Bigr)= h\Bigl(g\Bigl(\frac{1}{4}\Bigr)\Bigr) = \frac{17}{24}.$$
We leave it to the reader to check that if $\gamma : [\frac{1}{4},\frac{7}{12}) \to [\frac{1}{4},\frac{7}{12})$ is the
circle homeomorphism associated to $g$, $h$, and $s=\frac{1}{4}$, then $\gamma$ is conjugated to the homeomorphism in Theorem \ref{main} by $t \mapsto 3 t - \frac{3}{4}$.

Michele Triestino has indicated in personal communication that by combining arguments in this paper with earlier work,
he can show $F_{\frac{2}{3}}$ is not $C^2$-smoothable.

\subsection*{Some further remarks}

We have computed the rotation numbers of the homeomorphisms $f_{q,r}$ for $0 < q < 1$ having numerator and denominator
at most $30$ and $r$ having numerator and denominator at most $1000$.
In all cases $\rot(r_{q,r})$ is algebraic and has degree at most 2.
Notice that it is sufficient to consider $q < 1$ since $f_{\frac{1}{q},r}$ is conjugate to $f_{q,r}$ (via a rotation) and hence has
the same rotation number.

Our computations were done both using C code compiled with the {\tt gmp.h} library and Mathematica code.\footnote{
Our code is posted in GitHub at {\tt https://github.com/jimbelk/rotPLoS}.}
We have observed the following qualitative and quantitative
behavior when computing the rotation numbers of $f_{q,r}$ for rational $0 < q,r < 1$:
\begin{enumerate}

\item Individual computations of rotation numbers run instantaneously on a standard desktop
(64 bit Intel Core i7-7700 CPU running 3.60GHz with 8 processors);
the batch of calculations described above took about 100 processor hours (we divided the computation into several programs and
executed them simultaneously, with the computation being completed in around 24 hours).

\item For some values of $q$ --- for instance $q = \frac{2}{3}$ and $\frac{2}{7}$ ---
the rotation numbers of the $f_{q,r}$ can be computed without the need for
an arbitrary precision library (provided the denominator for $r$ was at most 100).
Other values of $q$ --- for instance $\frac{5}{6}$ and $\frac{7}{8}$ ---
seemed to have a tendency to produce particularly complex computations.

\item \label{rat_f_q,r}
For most $q$ which we examined, there was an $r$ such that $f_{q,r}$ had irrational rotation number.
The exceptions were
$q = \frac{2}{7}$,
$\frac{4}{7}$,
$\frac{2}{9}$, 
$\frac{2}{11}$,
$\frac{7}{11}$,
$\frac{8}{11}$,
$\frac{3}{13}$,
$\frac{7}{13}$,
$\frac{10}{13}$,
$\frac{3}{14}$,
$\frac{2}{15}$,
$\frac{3}{16}$,
$\frac{11}{16}$,
$\frac{2}{17}$,
$\frac{3}{17}$,
$\frac{4}{17}$,
$\frac{5}{17}$,
$\frac{8}{17}$,
$\frac{10}{17}$,
$\frac{11}{17}$, 
$\frac{13}{17}$,
$\frac{14}{17}$,    
$\frac{11}{18}$,
$\frac{2}{19}$,
$\frac{3}{19}$,
$\frac{4}{19}$,
$\frac{6}{19}$,
$\frac{7}{19}$,
$\frac{8}{19}$,
$\frac{10}{19}$,
$\frac{13}{19}$,
$\frac{16}{19}$, 
$\frac{3}{20}$,
$\frac{11}{20}$,
$\frac{2}{21}$,
$\frac{17}{21}$,
$\frac{3}{22}$,
$\frac{7}{22}$,
$\frac{19}{22}$,
$\frac{2}{23}$,
$\frac{3}{23}$,
$\frac{4}{23}$,
$\frac{5}{23}$, 
$\frac{6}{23}$,
$\frac{7}{23}$,
$\frac{9}{23}$,
$\frac{11}{23}$,
$\frac{14}{23}$,
$\frac{16}{23}$,
$\frac{17}{23}$,
$\frac{19}{23}$,
$\frac{20}{23}$,
$\frac{2}{25}$,
$\frac{3}{25}$,
$\frac{4}{25}$,
$\frac{7}{25}$,
$\frac{9}{25}$,
$\frac{11}{25}$,
$\frac{13}{25}$,
$\frac{19}{25}$,
$\frac{21}{25}$,
$\frac{3}{26}$,
$\frac{5}{26}$,
$\frac{7}{26}$,
$\frac{9}{26}$,
$\frac{17}{26}$,
$\frac{2}{27}$,
$\frac{4}{27}$,
$\frac{5}{27}$,
$\frac{7}{27}$,
$\frac{10}{27}$,
$\frac{17}{27}$,
$\frac{19}{27}$,
$\frac{23}{27}$,
$\frac{3}{28}$,
$\frac{19}{28}$,
$\frac{25}{28}$,
$\frac{2}{29}$,
$\frac{3}{29}$,
$\frac{4}{29}$,
$\frac{5}{29}$,
$\frac{7}{29}$,
$\frac{8}{29}$,
$\frac{9}{29}$,
$\frac{10}{29}$,
$\frac{14}{29}$,
$\frac{16}{29}$,
$\frac{17}{29}$,
$\frac{18}{29}$,
$\frac{19}{29}$,
$\frac{22}{29}$,
$\frac{23}{29}$,
$\frac{25}{29}$,
$\frac{26}{29}$,
$\frac{17}{30}$.
Together with the reciprocals of integers (which yield rational rotation numbers by \cite{GhysSergiescu}),
these 95 exceptions are the only values $q < 1$ having denominator at most 30 
such that $\rot(f_{q,r})$ is rational whenever $r$ has denominator at most 1000.
In fact typically, when there is an $r$ for which $\rot(f_{q,r})$ is irrational, $r$ can be taken to have a denominator of the same order of magnitude as $q$'s.
 
\item For each fixed $q$ which we examined, there were a small number of possible periodic parts of the continued fraction expansion of the rotation number of $f_{q,r}$.
For instance, when $q=\frac{6}{7}$ and the denominator of $r$ was at most 1000,
the periodic part was always $(1,2)$ and when $q = \frac{3}{8}$ the periodic part was always $(1,1,1,2)$.
With the exception of $q=\frac{7}{9}$ and $\frac{8}{9}$ for which we observed 13 and 28 periodic parts,
values of $q$ with a single digit numerator and denominator always generated 
fewer than 10 periodic parts for the continued fraction expansion of $\rot(f_{q,r})$ for $r$ having denominator at most
1000.
Note though that increasing the search range for $r$ from having denominator at most $500$ to having denominator at most
1000 did often increase the number of observed periodic parts.

\item Somewhat paradoxically, rational values of $\rot(f_{q,r})$ tended to be more complex than irrational values.
For instance if $q$ has single digit numerator and denominator, the longest continued fraction expansion of the form
$\rot(f_{q,r})$ for $r$ having denominator at most 1000 was always larger than the number of digits before the end
of the first period in the expansion of an irrational $\rot(f_{q,r})$ with the same constraints on $r$.

\end{enumerate}
Generally speaking, it would be interesting to provide explanations of these phenomena.
More specifically, these observations suggest the following questions.
\begin{question}
Are there rationals $q > 0$ which are not powers of an integer such that if $f \in \PLoS$ maps
$\Q$ to $\Q$ and has slopes powers of $q$, then the rotation number of $f$ is rational?\footnote{
  When this note was circulated originally, we asked specifically about the value $q=\frac{2}{7}$.
  Rainney Wan, an undergraduate at Cornell University working with the 3rd author has discovered counterexamples
  to this question when $q=2/7$.
  Still, it is unclear if there is a rational value of $r$ such that $f_{\frac{2}{7},r}$ has irrational
  rotation number.}
\end{question}

\begin{question}
For which rational $q$ does $F_q$ embed into $F$?
\end{question}

It could be that $F_q$ embeds into $F$ only when $q$ is a power of an integer.

\begin{question}
Is there an algorithm which determines for which rationals $q > 0$ there is a rational $r \in (0,1)$ such that
$f_{q,r}$ has irrational rotation number?
\end{question}

In the next question, $T_q$ is the circle analog of $F_q$.

\begin{question}
If $q > 0$ is rational, 
is there an algorithm which decides whether an element of $T_q$ has finite order?
\end{question}

\begin{question}
What are the possible irrational rotation numbers of elements of $\PLoS$ which map $\Q$ to $\Q$?
What if the slopes are required to be powers of a given rational $q > 0$?
\end{question}

\end{document}